\documentclass[12pt]{amsart}
\usepackage{amsfonts, amssymb, latexsym, hyperref, xcolor, tikz,tikz-cd,transparent}
\usepackage{ytableau}

\newtheorem{theorem}{Theorem}[section]
\newtheorem{proposition}[theorem]{Proposition}
\newtheorem{lemma}[theorem]{Lemma}

\newtheorem{Fact}[theorem]{Fact}

\newtheorem{claim}[theorem]{Claim}
\newtheorem*{claim*}{Claim}
\newtheorem{corollary}[theorem]{Corollary}
\newtheorem{Main Conjecture}[theorem]{Main Conjecture}

\theoremstyle{definition}

\theoremstyle{remark}

\newtheorem{example}[theorem]{Example}

\theoremstyle{plain}

\newcommand{\cellsize}{18}
\newlength{\cellsz} 
\newsavebox{\cell}
\sbox{\cell}{\begin{picture}(\cellsize,\cellsize)
\put(0,0){\line(1,0){\cellsize}}
\put(0,0){\line(0,1){\cellsize}}
\put(\cellsize,0){\line(0,1){\cellsize}}
\put(0,\cellsize){\line(1,0){\cellsize}}
\end{picture}}
\newcommand\cellify[1]{\def\thearg{#1}\def\nothing{}%
\ifx\thearg\nothing
\vrule width0pt height\cellsz depth0pt\else
\hbox to 0pt{\usebox{\cell} \hss}\fi%
\vbox to \cellsz{
\vss
\hbox to \cellsz{\hss$#1$\hss}
\vss}}
\newcommand\tableau[1]{\vtop{\let\\\cr
\baselineskip -16000pt \lineskiplimit 16000pt \lineskip 0pt
\ialign{&\cellify{##}\cr#1\crcr}}}

\newcommand{\Newton}{\mathrm{Newton}}

\newcommand{\excise}[1]{}

\begin{document}
\pagestyle{plain}
\title{Generalized Permutahedra and Schubert calculus}
\author{Avery St. Dizier}
\author{Alexander Yong}
\address{Dept.~of Mathematics, U.~Illinois at Urbana-Champaign, Urbana, IL 61801, USA} 
\email{stdizie2@illinois.edu, ayong@illinois.edu}
\date{October 27, 2020}

\begin{abstract}
We connect generalized permutahedra with Schubert calculus. Thereby, we give sufficient vanishing criteria for 
Schubert intersection numbers of the flag variety. 
Our argument utilizes recent developments in the study of Schubitopes, which 
are Newton polytopes of Schubert polynomials. 
The resulting tableau test executes in polynomial time. 
\end{abstract}

\maketitle
\section{Introduction}
\subsection{Background} Let $X={\sf Flags}({\mathbb C}^n)$ be the variety of complete flags of vector spaces 
\[F_{\bullet}: \langle 0\rangle \subset F_1\subset F_2\subset \cdots \subset F_i \subset \cdots \subset F_{n-1}\subset {\mathbb C}^n, \ \dim(F_i)=i.\] 
$X$ has a left-action of $GL_n$, and hence also by lower triangular invertible matrices $B_{-}$. The $B_-$-orbits
$X_w^{\circ}$ are indexed by permutations $w$ in the symmetric group $S_n$. Let $\leq$ denote Bruhat order. 
The \emph{Schubert varieties} are the
closures \[X_w=\coprod_{v\geq w} X_v^{\circ};\] 
this is codimension 
$\ell(w)=\#\{(i,j): 1\leq i<j\leq n, w(i)>w(j)\}$.
Thus, $X=X_{id}$ and $X_{w_0}$ is the \emph{Schubert point}, where $w_0=n \ n-1 \ n-2 \ \cdots 2 \ 1$.

The Poincar\'e duals $\sigma_w:=[X_w]$ form the \emph{Schubert basis} of $H^*(X)$, the cohomology ring
of~$X$. A \emph{Schubert problem} is 
$(w^{(1)}, w^{(2)},\ldots, w^{(k)})\in S_n^k$  with
$\sum_{i=1}^k \ell(w^{(i)})={n\choose 2}=\dim_{\mathbb C}(X)$.
The \emph{Schubert intersection number} is
\begin{align}\label{eqn:Oct22mmm}
C_{w^{(1)},w^{(2)},\ldots,w^{(k)}}:= & \text{\ multiplicity of $\sigma_{w_0}$ in\ }  \prod_{i=1}^k \sigma_{w^{(i)}}\in H^*(X)\\ 
= & \text{\ number of points in \ } \bigcap_{i=1}^k g_iX_{\sigma^{(i)}},\nonumber
\end{align}
where $(g_1,\ldots,g_k)$ are elements of a dense open subset ${\mathcal O}$ of $GL_n^k$ (whose existence is guaranteed
by Kleiman transversality). A textbook is \cite{Fulton}; expository papers include \cite{KL, Gillespie}.

Algorithms exist for computing these numbers; see, e.g., \cite{Billey:duke, Manivel, Knutson:recurrence} and the
references therein. It is the famous open problem of Schubert calculus to find a combinatorial counting rule that computes $C_{w^{(1)},w^{(2)},\ldots,w^{(k)}}$. Such a rule would generalize the classical \emph{Littlewood-Richardson rule}  
governing Schubert calculus of Grassmannians.

This paper explores a related, but not necessarily easier, open problem:
\begin{center}
Find an efficient algorithm to decide if
$C_{w^{(1)},w^{(2)},\ldots,w^{(k)}}=0$.
\end{center}

Known algorithms to compute $C_{w^{(1)},w^{(2)},\ldots,w^{(k)}}$ do not provide a solution (being inefficient).
In the Grassmannian setting, neither does the Littlewood-Richardson rule, \emph{per se}.
However,  the \emph{saturation theorem}
\cite{Knutson.Tao} permits a \emph{polynomial-time} algorithm in that case \cite{DeLoera, Mulmuley}, by way of linear programming results. For (generalized) flag varieties, criteria were found by A.~Knutson \cite{Knutson:descent} and K.~Purbhoo \cite{Purbhoo:root}; no efficiency guarantees were stated. 

\subsection{Vanishing criterion} Our main goal is to connect the theory of generalized permutahedra to Schubert calculus.
We give a sufficient test for $C_{w^{(1)},w^{(2)},\ldots,w^{(k)}}=0$ and prove 
it executes in polynomial-time. The starting point is a simple consideration about 
Schubert polynomials. However, it becomes
effective due to recent developments about Newton polytopes of
Schubert polynomials \cite{Fink.Meszaros.StDizier, MTY, ARY}, as instances of generalized permutahedra. 

The \emph{diagram} of $w\in S_n$, denoted $D(w)$, is the
subset of boxes of $[n]\times [n]$ given by
\[D(w):=\{(i,j):1\leq i,j\leq n, j<w(i), i<w^{-1}(j)\}.\]
Let 
${\sf code}(w)=(c_1(w),c_2(w),\ldots,c_n(w))$,
where $c_i$ counts boxes of $D(w)$ in row
$i$. Define 
\[D:=D(w^{(1)},\ldots,w^{(k)})\]
by concatenating $D(w^{(1)}),\ldots,D(w^{(k)})$, left to right.
Set ${\sf Tab}:={\sf Tab}_{w^{(1)},\ldots,w^{(k)}}$ to be the set of fillings of $D$ with nonnegative integers such that:
\begin{itemize}
\item[(a)] Each column is strictly increasing from top to bottom.
\item[(b)] Any label $\ell$ in row $r$ satisfies $\ell\leq r$.
\item[(c)] The number of $\ell$'s is $n-\ell$, for $1\leq \ell\leq n$.
\end{itemize}

The first version of our test is:
\begin{theorem}
\label{thm:main}
Let $(w^{(1)},\ldots,w^{(k)})$ be a Schubert problem. If ${\sf Tab}=\emptyset$ then $C_{w^{(1)},w^{(2)},\ldots,w^{(k)}}=0$.
There is an algorithm to  determine emptiness in $O({\sf poly}(n,k))$.
\end{theorem}

\begin{example}
\label{first:exa}
Let $w^{(1)}=3256147, w^{(2)}=2143657, w^{(3)}=4632175$. Below we depict $D$. The numerically 
labelled boxes are forced by conditions (a) and (b) for any (putative) $T\in {\sf Tab}$.
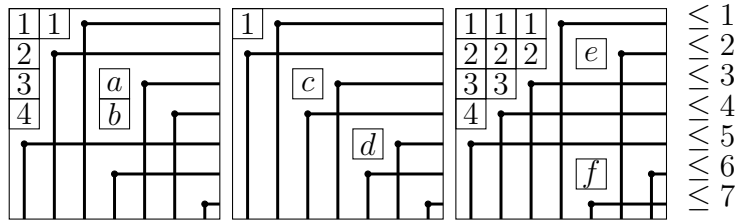
\begin{center}
\begin{tikzpicture}[scale=.4]
\draw (0,0) rectangle (7,7);
\draw (0,7) rectangle (1,6) node[pos=.5] {$1$};
\draw (1,7) rectangle (2,6) node[pos=.5] {$1$};
\draw (0,6) rectangle (1,5) node[pos=.5] {$2$};
\draw (0,5) rectangle (1,4) node[pos=.5] {$3$};
\draw (0,4) rectangle (1,3) node[pos=.5] {$4$};
\draw (3,5) rectangle (4,4) node[pos=.5] {$ a $};
\draw (3,4) rectangle (4,3) node[pos=.5] {$ b$};
\filldraw (2.5,6.5) circle (.5ex);
\draw[line width = .2ex] (2.5,0) -- (2.5,6.5) -- (7,6.5);
\filldraw (1.5,5.5) circle (.5ex);
\draw[line width = .2ex] (1.5,0) -- (1.5,5.5) -- (7,5.5);
\filldraw (4.5,4.5) circle (.5ex);
\draw[line width = .2ex] (4.5,0) -- (4.5,4.5) -- (7,4.5);
\filldraw (5.5,3.5) circle (.5ex);
\draw[line width = .2ex] (5.5,0) -- (5.5,3.5) -- (7,3.5);
\filldraw (0.5,2.5) circle (.5ex);
\draw[line width = .2ex] (0.5,0) -- (0.5,2.5) -- (7,2.5);
\filldraw (3.5,1.5) circle (.5ex);
\draw[line width = .2ex] (3.5,0) -- (3.5,1.5) -- (7,1.5);
\filldraw (6.5,0.5) circle (.5ex);
\draw[line width = .2ex] (6.5,0) -- (6.5,0.5) -- (7,0.5);
\end{tikzpicture}
\begin{tikzpicture}[scale=.4]
\draw (0,0) rectangle (7,7);
\draw (0,7) rectangle (1,6) node[pos=.5] {$1$};
\draw (2,5) rectangle (3,4) node[pos=.5] {$c $};
\draw (5,3) rectangle (4,2) node[pos=.5] {$d$};

\filldraw (1.5,6.5) circle (.5ex);
\draw[line width = .2ex] (1.5,0) -- (1.5,6.5) -- (7,6.5);
\filldraw (0.5,5.5) circle (.5ex);
\draw[line width = .2ex] (0.5,0) -- (0.5,5.5) -- (7,5.5);
\filldraw (3.5,4.5) circle (.5ex);
\draw[line width = .2ex] (3.5,0) -- (3.5,4.5) -- (7,4.5);
\filldraw (2.5,3.5) circle (.5ex);
\draw[line width = .2ex] (2.5,0) -- (2.5,3.5) -- (7,3.5);
\filldraw (5.5,2.5) circle (.5ex);
\draw[line width = .2ex] (5.5,0) -- (5.5,2.5) -- (7,2.5);
\filldraw (4.5,1.5) circle (.5ex);
\draw[line width = .2ex] (4.5,0) -- (4.5,1.5) -- (7,1.5);
\filldraw (6.5,0.5) circle (.5ex);
\draw[line width = .2ex] (6.5,0) -- (6.5,0.5) -- (7,0.5);
\end{tikzpicture}
\begin{tikzpicture}[scale=.4]
\draw (0,0) rectangle (7,7);
\draw (0,7) rectangle (1,6) node[pos=.5] {$1$};
\draw (1,7) rectangle (2,6) node[pos=.5] {$1$};
\draw (2,7) rectangle (3,6) node[pos=.5] {$1$};

\draw (0,6) rectangle (1,5) node[pos=.5] {$2$};
\draw (1,6) rectangle (2,5) node[pos=.5] {$2$};
\draw (2,6) rectangle (3,5) node[pos=.5] {$2$};
\draw (4,6) rectangle (5,5) node[pos=.5] {$e $};

\draw (0,5) rectangle (1,4) node[pos=.5] {$3$};
\draw (1,5) rectangle (2,4) node[pos=.5] {$3$};

\draw (0,4) rectangle (1,3) node[pos=.5] {$4$};

\draw (4,2) rectangle (5,1) node[pos=.5] {$f $};

\filldraw (3.5,6.5) circle (.5ex);
\draw[line width = .2ex] (3.5,0) -- (3.5,6.5) -- (7,6.5);
\filldraw (5.5,5.5) circle (.5ex);
\draw[line width = .2ex] (5.5,0) -- (5.5,5.5) -- (7,5.5);
\filldraw (2.5,4.5) circle (.5ex);
\draw[line width = .2ex] (2.5,0) -- (2.5,4.5) -- (7,4.5);
\filldraw (1.5,3.5) circle (.5ex);
\draw[line width = .2ex] (1.5,0) -- (1.5,3.5) -- (7,3.5);
\filldraw (0.5,2.5) circle (.5ex);
\draw[line width = .2ex] (0.5,0) -- (0.5,2.5) -- (7,2.5);
\filldraw (6.5,1.5) circle (.5ex);
\draw[line width = .2ex] (6.5,0) -- (6.5,1.5) -- (7,1.5);
\filldraw (4.5,0.5) circle (.5ex);
\draw[line width = .2ex] (4.5,0) -- (4.5,0.5) -- (7,0.5);
\end{tikzpicture}
\begin{tikzpicture}[scale=.4]
\node at (2.5,6.5) {$\leq 1$};
\node at (2.5,5.5) {$\leq 2$};
\node at (2.5,4.5) {$\leq 3$};
\node at (2.5,3.5) {$\leq 4$};
\node at (2.5,2.5) {$\leq 5$};
\node at (2.5,1.5) {$\leq 6$};
\node at (2.5,0.5) {$\leq 7$};
\end{tikzpicture}
\end{center}

Condition (b) forces
$e\leq 2, \ a,c \leq 3, \ b\leq 4,\  d\leq 5,  \ f\leq 6$.
Thus, to satisfy (c), $e=2$ is also forced, which implies $a,c=3$. So $T$ has at least five $3$'s, 
violating (c) for $\ell=3$.

Our idea (see Section~\ref{sec:4}) uses that $C_{w^{(1)},w^{(2)},w^{(3)}}=0$ if 
${\mathfrak S}_{w_0}=x_1^{6} x_2^5  x_3^4  x_4^3  x_5^2  x_6$ does not appear in the product of Schubert polynomials ${\mathfrak S}_{w^{(1)}} {\mathfrak S}_{w^{(2)}} {\mathfrak S}_{w^{(3)}}$, combined with  an argument that the rule of Theorem~\ref{thm:main} permits an efficient check of this vanishing condition. \qed
\end{example}

\subsection{Organization}
Section~\ref{sec:2} discusses generalized permutahedra; we derive facts we will use. Section~\ref{sec:3} reviews the subfamily of Schubitopes. In Section~\ref{sec:4} we state Theorem~\ref{thm:variation}, an ``asymmetric'' version of Theorem~\ref{thm:main}; it is a stronger
test, see Proposition~\ref{prop:Oct4yyy}. Theorem~\ref{thm:schubnec} gives 
linear inequalities necessary for $C_{w^{(1)},\ldots,w^{(k)}}>0$.  Theorems~\ref{thm:main},~\ref{thm:variation},~\ref{thm:schubnec}, and Proposition~\ref{prop:Oct4yyy} are proved together, as they follow from the same reasoning. In Section~\ref{sec:5}, we 
compare with the vanishing criteria of \cite{Knutson:descent} and \cite{Purbhoo:root}. 
We show examples that our test captures but are not captured by those criteria, and conversely.

\section{Newton Polytopes of products}\label{sec:2}

If $f$ is an element of a polynomial ring whose variables are indexed by some set $I$,
the \emph{support} of $f$ is the lattice point set in $\mathbb R^I$ consisting of the exponent vectors of the monomials that have nonzero coefficient in~$f$.
The \emph{Newton polytope} 
$\Newton(f)\subseteq\mathbb R^I$ 
is the convex hull of the support of~$f$. A polynomial $f$ has \emph{saturated Newton polytope} (\emph{SNP}) if every lattice point in $\Newton(f)$ is a vector in the support of $f$ \cite{MTY}. 

The \emph{standard permutahedron} is the polytope in $\mathbb{R}^n$ whose vertices consist of all permutations of the entries of the vector $(0,1,\ldots,n-1)$. A \emph{generalized permutahedron} is a deformation of the standard permutahedron obtained by translating the vertices in such a way that all edge directions and orientations are preserved (edges are allowed to degenerate to points). Generalized permutahedra are uniquely parametrized by \emph{submodular functions} (see \cite[Theorem 12.3]{Aguiar:hopf}). These are maps
\[z:2^{[n]}\to\mathbb{R},\] 
such that $z_\emptyset=0$ and 
\[z_I+z_J\geq z_{I\cup J}+z_{I\cap J}  \text{ \ for all $I,J\subseteq [n]$.}\] 
Given $z$, the associated generalized permutahedron is given by
\[P(z)=\left \{t\in\mathbb{R}^n: \sum_{i\in I}{t_i}\leq z_I \mbox{ for } I\neq [n], \mbox{ and } \sum_{i=1}^{n}{t_i}=z_{[n]}  \right \}. \] 

The vertices of generalized permutahedra have been determined. 
\begin{proposition}[{\cite[Corollary 44.3a]{Schrijver}}]
\label{prop:vertex}
	Let $P(z)$ be a generalized permutahedron in $\mathbb{R}^n$. The vertices of $P(z)$ are 
	$\{v(w): w\in S_n  \} $
	where $v(w)=(v_1,\ldots,v_n)\in {\mathbb R}^n$ is defined by
	\begin{equation}
	\label{eqn:Oct21xyz}
	v_{w_k}=z_{{\{w_1,\ldots,w_k\}}} -z_{{\{w_1,\ldots,w_{k-1}\}}}.
	\end{equation}
\end{proposition}

It is well-known that the class of generalized permutahedra is closed under Minkowski sums (see for instance \cite[Lemma 2.2]{Ardila}). We provide a proof for completeness.
\begin{lemma}
	\label{lem:minksum}
	If $P(z)$ and $P(z')$ are generalized permutahedra, then 
	\[P(z)+P(z')=P(z+z').\]
\end{lemma}
\begin{proof}
	Clearly $P(z)+P(z')\subseteq P(z+z')$. For the opposite containment, let $q$ be a vertex of $P(z+z')$. 
	By Proposition~\ref{prop:vertex}, write $q$ in the form $q=v(w)$ for some $w\in S_n$. Let $p$ and $p'$ be the vertices of $P(z)$ and $P(z')$ respectively corresponding to $w$. By (\ref{eqn:Oct21xyz}), $q=p+p'\in P(z)+P(z')$. Convexity implies $ P(z+z')\subseteq P(z)+P(z')$.
\end{proof}

It follows easily from \cite[Theorem 46.2]{Schrijver} that whenever $z$ and $z'$ are integer-valued, $P(z)\cap P(z')$ is either empty or an integral polytope (all vertices are lattice points). This is used to prove that 
\emph{integer polymatroids} \cite[Chapter~44]{Schrijver} satisfy a generalization of the \emph{integer decomposition property}. We state and prove (for convenience) the special case that applies to generalized permutahedra:
 
\begin{theorem}[{\cite[Corollary 46.2c]{Schrijver}}]
	\label{thm:gidp}
	If $P(z)$ and $P(z')$ are integral generalized permutahedra in $\mathbb{R}^n$, then
	\[(P(z)\cap \mathbb{Z}^n) + (P(z')\cap \mathbb{Z}^n) = (P(z)+P(z'))\cap \mathbb{Z}^n.  \]
\end{theorem}
\begin{proof}
Let $r\in(P(z)+P(z'))\cap \mathbb{Z}^n$. Set $Q=r+(-1)P(z')$. Clearly, $Q$ is a generalized permutahedron (by the deformation description). Also note that $r=p+p'$ for some $p\in P(z)$ and $p'\in P(z')$, so $p\in P\cap Q$ and $P\cap Q\neq \emptyset$. Since both $r$ and $z'$ are integral, $Q$ is an integral polytope. Thus $P\cap Q$ contains an integer point $q$. By definition of $Q$, the lattice point $r-q$ is in $P(z')$. Finally, we have 
\[r=q+(r-q)\in (P(z)\cap \mathbb{Z}^n)+(P(z')\cap \mathbb{Z}^n).\qedhere\]
\end{proof}

Therefore, in the realm of generalized permutahedra, SNP carries through products.

\begin{proposition}
\label{prop:genfact}
	If $f,g\in\mathbb{R}_{\geq 0}[x_1,\ldots,x_n]$ have SNP and $\Newton(f),\Newton(g)$ are generalized permutahedra then 
	\begin{enumerate}
		\item[(i)] $\Newton(fg)$ is a generalized permutahedron;
		\item[(ii)] $fg$ has SNP.
	\end{enumerate}
\end{proposition}

\begin{proof}
	For any polynomials $f$ and $g$,
	$\Newton(fg)=\Newton(f)+\Newton(g)$. 
	Statement (i) follows from Lemma \ref{lem:minksum}. Statement (ii) follows from Lemma \ref{lem:minksum} and Theorem \ref{thm:gidp}.
\end{proof}

\section{Schubitopes, and an integer linear program}\label{sec:3}

We are interested in a particular family of generalized permutahedra.
For $D\subseteq [n]\times [m]$, the \emph{Schubitope} ${\mathcal S}_D$ was defined 
by C.~Monical, N.~Tokcan, and the second author \cite{MTY}.
Fix $S\subseteq [n]$ and a column $c \in [m]$. Let $\omega_{c,S}(D)$ be formed by reading $c$ from top to bottom and recording 
\begin{itemize} 
\item $($ if $(r,c) \notin D$ and $r \in S$,  
\item $)$ if $(r,c) \in D$ and $r \notin S$, and
\item $\star$ if $(r,c) \in D$ and $r \in S$.
\end{itemize}
Let 
\[\theta_{D}^c(S)= \#\text{paired $(\ )$'s in $\omega_{c,S}({D})$} + 
\#\text{$\star$'s in $\omega_{c,S}({D})$}.\]  
Set 
$\theta_{D}(S) = \sum_{c \in [n]} \theta_{D}^c(S)$. Define the \emph{Schubitope} as  
\[{\mathcal S}_{D}=\left\{(\alpha_1,\ldots,\alpha_n)\in {\mathbb R}_{\geq 0}^n:
\sum_{i=1}^n \alpha_i=\#{D} \text{\ and \ }
\sum_{i\in S}\alpha_i \leq \theta_{D}(S) \text{\ for all $S \subset [n]$}\right\}.\]

\begin{example}[\emph{cf.}~{\cite[Section~1]{MTY}}]
	Let $w=21543$. The Schubert polynomial of $w$ is
	\begin{align*}
		\mathfrak{S}_w=x_1^3 x_2+x_1^3 x_3&+x_1^3x_4+x_1^2x_2^2+x_1^2x_3^2+2x_1^2x_2 x_3+x_1^2 x_2 x_4 +x_1^2 x_3 x_4\\&+x_1 x_2 x_3^2+x_1 x_2^2 x_3 +x_1 x_2^2 x_4 +x_1 x_3^2 x_4 +x_1 x_2 x_3 x_4 .
	\end{align*}
	As stated in Theorem~\ref{thm:isSNP}, $\mathcal{S}_{D(w)}=\mathrm{Newton}(\mathfrak{S}_w)$. This generalized
	permutahedron and a minimal set of defining inequalities are shown in Figure \ref{fig:newtonexample}.\qed
\end{example}
\begin{figure}[h]
	\centering
	\begin{minipage}{.5\linewidth}
		\flushleft
		\begin{tikzpicture}%
		[x={(-0.358390cm, -0.633497cm)},
		y={(0.794712cm, 0.178421cm)},
		z={(-0.489886cm, 0.752893cm)},
		scale=1.75,
		back/.style={densely dotted, thick},
		edge/.style={color=orange, thick},
		facet/.style={fill=green,fill opacity=0.4},
		vertex/.style={inner sep=1pt,circle,draw=blue!25!black,fill=blue!75!black,thick,anchor=base}]
		\coordinate (1, 0, 2) at (1, 0, 2);
		\coordinate (3, 1, 0) at (3, 1, 0);
		\coordinate (1, 1, 2) at (1, 1, 2);
		\coordinate (1, 2, 0) at (1, 2, 0);
		\coordinate (1, 2, 1) at (1, 2, 1);
		\coordinate (3, 0, 1) at (3, 0, 1);
		\coordinate (2, 0, 2) at (2, 0, 2);
		\coordinate (3, 0, 0) at (3, 0, 0);
		\coordinate (2, 2, 0) at (2, 2, 0);
		\draw[edge,back] (1, 0, 2) -- (1, 2, 0);
		\draw[edge,back] (1, 0, 2) -- (3, 0, 0);
		\draw[edge,back] (1, 2, 0) -- (3, 0, 0);
		\node[vertex] at (1, 1, 1)     {};
		\node[vertex] at (2, 0, 1)     {};
		\node[vertex] at (2, 1, 0)     {};
		\fill[facet] (1, 1, 2) -- (1, 2, 1) -- (2, 2, 0) -- (3, 1, 0) -- (3, 0, 1) -- (2, 0, 2) -- cycle {};
		\fill[facet] (2, 0, 2) -- (1, 0, 2) -- (1, 1, 2) -- cycle {};
		\fill[facet] (2, 2, 0) -- (1, 2, 0) -- (1, 2, 1) -- cycle {};
		\fill[facet] (3, 0, 0) -- (3, 1, 0) -- (3, 0, 1) -- cycle {};
		\draw[edge] (1, 0, 2) -- (1, 1, 2);
		\draw[edge] (1, 0, 2) -- (2, 0, 2);
		\draw[edge] (3, 1, 0) -- (3, 0, 1);
		\draw[edge] (3, 1, 0) -- (3, 0, 0);
		\draw[edge] (3, 1, 0) -- (2, 2, 0);
		\draw[edge] (1, 1, 2) -- (1, 2, 1);
		\draw[edge] (1, 1, 2) -- (2, 0, 2);
		\draw[edge] (1, 2, 0) -- (1, 2, 1);
		\draw[edge] (1, 2, 0) -- (2, 2, 0);
		\draw[edge] (1, 2, 1) -- (2, 2, 0);
		\draw[edge] (3, 0, 1) -- (2, 0, 2);
		\draw[edge] (3, 0, 1) -- (3, 0, 0);
		\node[vertex] at (1, 0, 2)     {};
		\node[vertex] at (3, 1, 0)     {};
		\node[vertex] at (1, 1, 2)     {};
		\node[vertex] at (1, 2, 0)     {};
		\node[vertex] at (1, 2, 1)     {};
		\node[vertex] at (3, 0, 1)     {};
		\node[vertex] at (2, 0, 2)     {};
		\node[vertex] at (3, 0, 0)     {};
		\node[vertex] at (2, 2, 0)     {};
		\node[vertex] at (2, 1, 1)     {};
		\node[left]  at (1, 0, 2)     {$x_1x_3^2x_4$};
		\node[below] at (3, 1, 0)     {$x_1^3x_2$};
		\node[above] at (1, 1, 2)     {$x_1x_2x_3^2$};
		\node[right] at (1, 2, 0)     {$x_1x_2^2x_4$};
		\node[right] at (1, 2, 1)     {$x_1x_2^2x_3$};
		\node[left]  at (3, 0, 1)     {$x_1^3x_3$};
		\node[left]  at (2, 0, 2)     {$x_1^2x_3^2$};
		\node[left]  at (3, 0, 0)     {$x_1^3x_4$};
		\node[right] at (2, 2, 0)     {$x_1^2x_2^2$};
		\node[above] at (2, 1, 1)     {$x_1^2x_2x_3$};
		\node[above] at (1, 1, 1)     {\color{black}{\transparent{0.6}{$x_1x_2x_3x_4$}}};
		\node[above] at (2, 0, 1)     {\color{black}{\transparent{0.6}{$x_1^2x_3x_4$}}};
		\node[above] at (2, 1, 0)     {\color{black}{\transparent{0.6}{$x_1^2x_2x_4$}}};
		\end{tikzpicture}
	\end{minipage}
	\begin{minipage}{.49\linewidth}
		\flushright
		\begin{align*}
			\alpha_1&\leq 3\\
			\alpha_2&\leq 2\\
			\alpha_3&\leq 2\\
			\alpha_4&\leq 1\\
			\alpha_1+\alpha_2+\alpha_3&\leq 4\\
			\alpha_1+\alpha_2+\alpha_4&\leq 4\\
			\alpha_1+\alpha_3+\alpha_4&\leq 4\\
			\alpha_2+\alpha_3+\alpha_4&\leq 3\\
			\alpha_1+\alpha_2+\alpha_3+\alpha_4&=4
		\end{align*}
	\end{minipage}
	\caption{${\mathcal S}_{D(21543)}=\Newton(\mathfrak{S}_{21543})$ and a minimal set of defining inequalities.}
	\label{fig:newtonexample}
\end{figure}
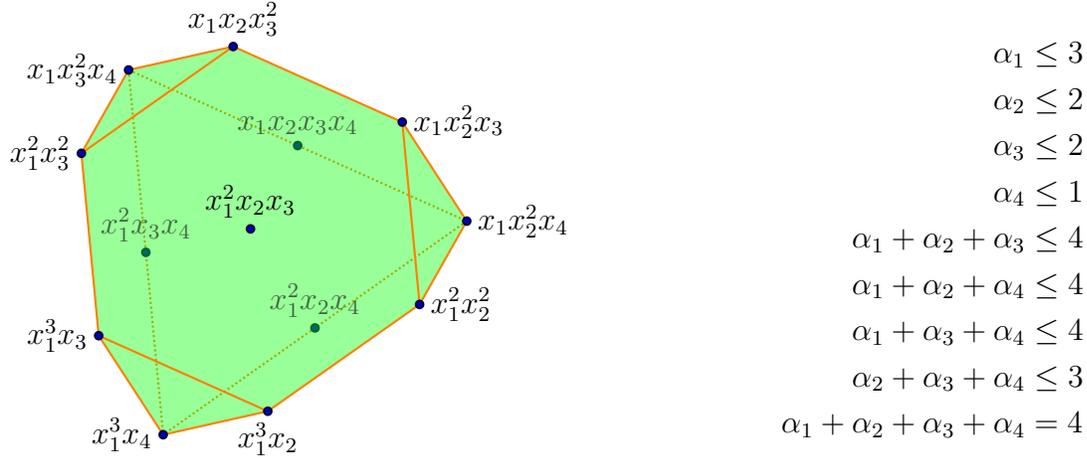

	We need some notions from A.~Adve, C.~Robichaux, and the second author's paper \cite{ARY}.  Given 
	$D \subseteq [n]\times [m]$ and $\alpha = (\alpha_1,\dots,\alpha_n) \in \mathbb{Z}_{\geq 0}^n$. Let 
	\[\mathcal{P}(D,\alpha) \subseteq \mathbb{R}^{n\times m}\]
	be the polytope whose points
	\[
	(\alpha_{ij})_{1\leq i\leq n, 1\leq j\leq m} = (\alpha_{11},\dots,\alpha_{n1},\dots,\alpha_{1m},\dots,\alpha_{nm})
	\]
	satisfy the inequalities (I),(II),(III) below.
	\begin{enumerate}
		\item[(I)] Column-injectivity: For all $i,j \in [n]$,
		\begin{align*}
		0 \leq \alpha_{ij} \leq 1.
		\end{align*}
		
		\item[(II)] Content: For all $i \in [n]$,
		\begin{align*}
		\sum_{j=1}^n \alpha_{ij} = \alpha_i.
		\end{align*}
		
		\item[(III)] Row bounds: For all $s,j \in [n]$,
		\begin{align*}
		\sum_{i=1}^s \alpha_{ij} \geq \#\{(i,j) \in D : i \leq s\}.
		\end{align*}
	\end{enumerate}

\excise{	
	\begin{proposition}\label{prop:column_content}
		Let $D \subseteq [n]^2$ and $\alpha = (\alpha_1,\dots,\alpha_n) \in \mathbb{Z}^n_{\geq 0}$ with $\alpha_1 + \dots + \alpha_n = \#D$. If $(\alpha_{ij}) \in \mathcal{P}(D,\alpha)$, then for each $j \in [n]$, we have
		\begin{align*}
		\sum_{i=1}^{n} \alpha_{ij} = \#\{(i,j) \in D : i \in [n]\}.
		\end{align*}
	\end{proposition}
}

Define ${\sf Tab}(D,\alpha)$ to be the set of fillings of $D$ with nonnegative integers such that
\begin{itemize}
\item[(a)] Each column is strictly increasing from top to bottom.
\item[(b)] Any label $\ell$ in row $r$ satisfies $\ell\leq r$.
\item[(c)] The number of $\ell$'s is $\alpha_{\ell}$.
\end{itemize}
 
	\begin{theorem}[{\cite{ARY}}]\label{thm:ARY1}
	Suppose $D \subseteq [n]\times [m]$ and $\alpha = (\alpha_1,\dots,\alpha_n) \in \mathbb{Z}_{\geq 0}^n$. 
	Then 
	\[\alpha\in {\mathcal S}_D\iff {\sf Tab}(D,\alpha)\neq\emptyset.\]
	
		The map 
		$f:{\sf Tab}(D,\alpha)\to {\mathcal P}(D,\alpha)$,
		that sets $\alpha_{ij}=1$ if the label $i$ appears in column $j$ of $D$, and set $\alpha_{ij}=0$ otherwise, is a bijection.
		Therefore ${\sf Tab}(D,\alpha) \neq \emptyset$ if and only if $\alpha_1 + \dots + \alpha_n = \#D$ and $\mathcal{P}(D,\alpha) \cap \mathbb{Z}^{n^2} \neq \emptyset$.
	\end{theorem}

	\begin{theorem}[{\cite{ARY}}]\label{thm:relaxation_equivalence}
		Let $D \subseteq [n]\times [m]$ and $\alpha = (\alpha_1,\dots,\alpha_n) \in \mathbb{Z}^n$ with $\alpha_1 + \dots + \alpha_n = \#D$. Then $\mathcal{P}(D,\alpha) \cap \mathbb{Z}^{n\times m} \neq \emptyset$ if and only if $\mathcal{P}(D,\alpha) \neq \emptyset$. 
	\end{theorem}

The above two theorems, combined with the ellipsoid method and/or interior point methods in linear programming, implies:
\begin{corollary}[\cite{ARY}]
\label{cor:runtime}
Deciding if $\alpha\in {\mathcal S}_{D}$, or equivalently, if ${\sf Tab}(D,\alpha)=\emptyset$,
 can be determined in $O({\sf poly}(n,m))$-time.
\end{corollary}
As explained in \cite{ARY}, by using the codes of $w^{(i)}$ as the encoding of the decision problem, or ``compressing'' $D$, one can reduce the upper bound on the complexity. We will not describe these technical improvements here, although they may be applied.

\section{Schubert polynomials and Schubitopes}\label{sec:4}

\subsection{Schubert polynomials}
Our reference for \emph{Schubert polynomials} is \cite{Manivel}. They are recursively defined; the initial condition is that
for $w_0\in S_n$, 
\[{\mathfrak S}_{w_0}:=x_1^{n-1}x_2^{n-2}\cdots x_{n-1}.\] 
The \emph{divided difference operator} on polynomials in ${\sf Pol}:={\mathbb Z}[x_1,x_2,\ldots]$ is
\[\partial_i:{\sf Pol}\to {\sf Pol}, \ f\mapsto \frac{f(\ldots,x_i,x_{i+1},\ldots)-f(\ldots,x_{i+1},x_i,\ldots)}{x_i-x_{i+1}}.\]
If $w\neq w_0$, let $i$ satisfy $w(i)<w(i+1)$, then 
${\mathfrak S}_w:=\partial_i {\mathfrak S}_{ws_i}$.  Since the divided difference operators satisfy the braid
relations
\[\partial_i\partial_j=\partial_j\partial_i \text{ for $|i-j|\geq 2$}; \ \  \partial_i \partial_{i+1} \partial_i = \partial_{i+1}\partial_i\partial_{i+1},\] 
it follows that ${\mathfrak S}_w$
only depends on $w$, and not the choices of $i$ in the recursion.

Schubert polynomials are stable under the inclusion of $S_n\hookrightarrow S_{n+1}$ that sends $w$ to 
$w$ with $n+1$ appended. Thus, one defines ${\mathfrak S}_w$ for 
$w\in S_{\infty}=\bigcup_{n\geq 1} S_n$.
The set of Schubert polynomials $\{{\mathfrak S}_w: w\in S_{\infty}\}$ forms a ${\mathbb Z}$-linear basis of ${\sf Pol}$.

Borel's isomorphism \cite[Chapter~9; Prop.~3]{Fulton} asserts
\[H^*(X)\cong {\mathbb Q}[x_1,\ldots,x_n]/I^{S_n}
\text{ \ where $I^{S_n}=\langle e_d(x_1,\ldots,x_n): 1\leq d\leq n\rangle$},\]
and
\[e_{d}(x_1,\ldots,x_n)=\sum_{1\leq i_1<i_2<\cdots<i_d\leq n}x_{i_1} x_{i_2}\cdots  x_{i_d}\]
is the $d$-th \emph{elementary symmetric polynomial}. Under this isomorphism,
\begin{equation}\label{eqn:Oct22aaa}
\sigma_w\mapsto {\mathfrak S}_w+I^{S_n}.
\end{equation}
One has the \emph{polynomial} identity
\[{\mathfrak S}_u {\mathfrak S}_v=\sum_{w\in S_{\infty}} C_{u,v}^w{\mathfrak S}_w \in {\sf Pol}.\]

Define $C_{w^{(1)},\ldots,w^{(k-1)}}^{w^{(k)}}$ to be the multiplicity of $\sigma_{w^{(k)}}$ in $\prod_{i=1}^{k-1}\sigma_{w^{(i)}}\in H^*(X)$, which we also write with the coefficient operator as $[\sigma_{w^{(k)}}] \  \prod_{i=1}^{k-1}\sigma_{w^{(i)}}$.
 
\begin{lemma}
\label{lemma:symmtoasymm}
$C_{w^{(1)},\ldots,w^{(k-1)}}^{w^{(k)}}=C_{w^{(1)},\ldots,w^{(k-1)},w_0w^{(k)}}$. Also,
$C_{w^{(1)},\ldots,w^{(k-1)}}^{w^{(k)}}=[{\mathfrak S}_{w^{(k)}}] \ \prod_{i=1}^{k-1}{\mathfrak S}_{w^{(i)}}$. In particular
$C_{u,v}^w=C_{u,v,w_0w}$.
\end{lemma}
\begin{proof}
Duality in Schubert calculus (see, e.g., \cite[Proposition~3.6.11]{Manivel}) states that if $\ell(u)+\ell(v)={n\choose 2}$ then 
\[\sigma_u\smile \sigma_v=
\begin{cases}
\sigma_{w_0} & \text{if $v=w_0u$}\\
0 & \text{otherwise}.
\end{cases}\]
Now, 
\[\prod_{i=1}^{k-1}\sigma_{w^{(i)}}=C_{w^{1},\ldots,w^{(k-1)}}^{w^{(k)}}\sigma_{w^{(k)}}+\sum_{w\in S_n, w\neq w^{(k)}} C_{w^{(1)},\ldots,w^{(k-1)}}^w \sigma_w.\]
Multiply both sides by $\sigma_{w_0 w^{(k)}}$ and apply duality. Then use (\ref{eqn:Oct22mmm}) to obtain the first
statement. The second assertion follows from (\ref{eqn:Oct22aaa}). The final claim is merely the $k=3$ case.
\end{proof}

\begin{lemma} 
\label{lemma:Oct2zzz}
If $(w^{(1)},\ldots,w^{(k)})$ is a
Schubert problem then 
\[C_{w^{(1)},\ldots,w^{(k)}}=[x_1^{n-1}x_2^{n-2}\cdots x_{n-1}]\ \prod_{i=1}^k {\mathfrak S}_{w^{(i)}}.\] 
\end{lemma}
\begin{proof}
This follows from (\ref{eqn:Oct22mmm}), (\ref{eqn:Oct22aaa}),  and 
${\mathfrak S}_{w_0}=x_1^{n-1}x_2^{n-2}\cdots x_{n-1}$.
\end{proof}

\subsection{Schubitopes are Newton polytopes}
This result from work of A.~Fink, K.~M\'esz\'aros, and the first author \cite{Fink.Meszaros.StDizier} proves conjectures of \cite{MTY}: 
\begin{theorem}[{\cite[Theorems~7,10]{Fink.Meszaros.StDizier}}]
\label{thm:isSNP}
${\mathcal S}_{D(w)}=\Newton({\mathfrak S}_w)$, and ${\mathfrak S}_w$ has SNP.
\end{theorem}

\begin{theorem}[{\cite[Corollary~8]{Fink.Meszaros.StDizier}}]
\label{thm:genperm}
${\mathcal S}_{D(w)}$ is a generalized permutahedron.
\end{theorem}

\begin{proposition}
\label{prop:Oct2}
$f=\prod_{i=1}^{k-1} {\mathfrak S}_{w^{(i)}}$ has SNP. In addition, 
\begin{equation}
\label{eqn:Oct10fff}
\Newton(f)=\sum_{i=1}^{k-1} {\mathcal S}_{D(w^{(i)})} {\text{ \ (Minkowski sum).}}
\end{equation}
\end{proposition}
\begin{proof}
This follows from combining Theorems~\ref{thm:isSNP} and~\ref{thm:genperm} with Proposition~\ref{prop:genfact}.
 \end{proof}

By the same argument as Proposition~\ref{eqn:Oct10fff}, any product of \emph{key polynomials} (see, e.g., \cite{Reiner.Shimozono}) with Schubert polynomials
is SNP, and has  a similarly described Newton polytope.

\begin{corollary}
\label{cor:Oct11hhh}
If $\alpha\in {\mathbb Z}_{\geq 0}^n$ then 
\[[x^{\alpha}]\prod_{i=1}^{k-1} {\mathfrak S}_{w^{(i)}}\neq 0 \iff
\alpha\in \sum_{i=1}^{k-1} {\mathcal S}_{D(w^{(i)})}.\]
\end{corollary}
\begin{proof}
Let $f=\prod_{i=1}^{k-1} {\mathfrak S}_{w^{(i)}}$. If 
$[x^{\alpha}] f\neq 0$ then 
$\alpha\in \Newton(f)$. Now apply (\ref{eqn:Oct10fff}).
Conversely, by (\ref{eqn:Oct10fff}), $\alpha\in \Newton(f)$. By Proposition~\ref{prop:Oct2}, $f$ has SNP. 
Hence $[x^{\alpha}] f\neq 0$.
\end{proof}

\subsection{The asymmetric version of Theorem~\ref{thm:main}} 
Let 
$D':=D(w^{(1)},\ldots,w^{(k-1)})$ and let 
${\sf Tab}':={\sf Tab}'_{w^{(1)},\ldots,w^{(k)}}$ be the set of fillings of $D'$ with nonnegative integers such that:
\begin{itemize}
\item[(a)] Each column is strictly increasing from top to bottom.
\item[(b)] Any label $\ell$ in row $r$ satisfies $\ell\leq r$.
\item[(c)] The number of $\ell$'s is $c_{\ell}(w^{(k)})$.
\end{itemize}

\begin{theorem}
\label{thm:variation}
Let $(w^{(1)},\ldots,w_0 w^{(k)})$ be a Schubert problem. If ${\sf Tab}'=\emptyset$ then 
$C_{w^{(1)},w^{(2)},\ldots,w^{(k-1)}}^{w^{(k)}}=0$.
There is an algorithm to  determine emptiness in $O({\sf poly}(n,k))$.
\end{theorem}


\begin{proposition}
\label{prop:Oct4yyy}
If Theorem~\ref{thm:main}'s test shows $C_{w^{(1)},\ldots,w^{(k-1)},w_0w^{(k)}}=0$ then Theorem~\ref{thm:variation}'s 
test also shows $C_{w^{(1)},\ldots,w^{(k-1)}}^{w^{(k)}}=0$.
\end{proposition}

\begin{example}
The converse of Proposition~\ref{prop:Oct4yyy} is false. That is,
Theorem~\ref{thm:variation} provides a strictly stronger test than Theorem~\ref{thm:main}. For example, 
\[{\mathfrak S}_{4123}{\mathfrak S}_{1342}=x_1^{4}{x_3}+x_1^{4}{x_2}+x_1^{3}{x_2}{x_3}\]
avoids ${\sf code}(4312)=3200$ as an exponent vector, proving 
$C_{u,v}^w=C_{4123,1342}^{4312}=0$.
However,
\[{\mathfrak S}_{u}{\mathfrak S_v}{\mathfrak S}_{w_0w}=x_1^4x_2^{2}+x_1^{4}x_3^{2}+3x_1^{4}{x_2}{x_3}+x_1^{3}{x_2}x_3^{2}+\underline{x_1^{3}x_2^{2}{x_3}}+x_1^{5}{ x_3}+x_1^{5}{x_2}\]
implies ${\sf Tab}\neq \emptyset$, and hence Theorem~\ref{thm:main} does not show $C_{u,v,w_0w}=C_{4123,1342,1243}=0$.\qed
\end{example}

\subsection{The Schubitope inequalities and Schubert calculus}\label{sec:5}

The Schubitope inequalities provide necessary conditions for nonvanishing of a Schubert
intersection number. 

\begin{theorem}
\label{thm:schubnec}
If $C_{w^{(1)},w^{(2)},\ldots, w^{(k)}}>0$ then $(n-1,n-2,\ldots,2,1)$ must satisfy 
the Schubitope inequalities defining ${\mathcal S}_{D}$ where $D=D(w^{(1)},\ldots,w^{(k)})$. Similarly, if
$C_{w^{(1)},\ldots,w^{(k-1)}}^{w^{(k)}}>0$ then ${\sf code}(w^{(k)})$ must satisfy the Schubitope inequalities
defining ${\mathcal S}_{D'}$ where $D'=D(w^{(1)},\ldots,w^{(k-1)})$.
\end{theorem}

Let 
\[s_{\lambda}(x_1,\ldots,x_k)=\sum_{T} x^T\] be the \emph{Schur polynomial} of $\lambda$, 
where the sum is over semistandard Young tableaux of shape $\lambda$ filled using $\{1,2,\ldots,k\}$ and
$x^T=\prod_{i=1}^k x_i^{\#i\in T}$. Then
\[s_{\lambda}(x_1,\ldots,x_k)s_{\mu}(x_1,\ldots,x_k)=\sum_{\nu} c_{\lambda,\mu}^{\nu} s_{\nu}(x_1,\ldots,x_k),\]
where $c_{\lambda,\mu}^{\nu}$ is the \emph{Littlewood-Richardson coefficient}. By the proof of \cite[Proposition~2.9]{MTY},
\begin{equation}
\label{eqn:Oct19fff}
x^{\nu}\in s_{\lambda}s_{\mu} \text{ \ if and only if $\nu\in \Newton(s_{\lambda+\mu})={\mathcal P}_{\lambda+\mu}$ (the permutahedron for $\lambda+\mu$).}
\end{equation}
By Rado's theorem \cite[Theorem~1]{Rado}, this means $\nu\leq_{\rm Dom} \lambda+\mu$ (dominance order). That is
\[c_{\lambda,\mu}^{\nu}>0 \implies \sum_{i=1}^t \nu_i\leq \sum_{j=1}^t \lambda_j +\sum_{k=1}^t \nu_k, \text{ for $t\geq 1$}.\]
These are instances of the famous \emph{Horn's inequalities}; see the survey \cite{Fulton:horn}. 
(Those are generalized in the ``Levi-movable'' case of $X$ in work of P.~Belkale-S.~Kumar \cite{BK}.) Our methods are
in the same vein. Hence, we speculate Theorem~\ref{thm:schubnec} is a first glimpse of putative linear inequalities that control  $C_{w^{(1)},\ldots,w^{(k)}}>0$. We hope to study this further in a sequel.
%

\subsection{Proof of Theorems~\ref{thm:main}, \ref{thm:variation}, \ref{thm:schubnec} 
and Proposition~\ref{prop:Oct4yyy}:} \label{sec:theproofs}
We combine the proofs of these four results since they all stem from the same reasoning.

We prove Theorem~\ref{thm:variation} first.
It is known (\emph{e.g.}, follows from \cite[Theorem~2.5.1]{Manivel}) that
\begin{equation}
\label{eqn:hasterm}
[x^{{\sf code}(w)}]{\mathfrak S}_w\neq 0.
\end{equation}
Hence 
\begin{align*}
[x^{{\sf code}(w^{(k)})}]\prod_{i=1}^{k-1}{\mathfrak S}_{w^{(i)}}=0 \Rightarrow C_{w^{(1)},w^{(2)},\ldots,w^{(k-1)}}^{w^{(k)}}=0.
\end{align*}
 By one direction of Corollary~\ref{cor:Oct11hhh},
\begin{equation}
[x^{{\sf code}(w^{(k)})}]\prod_{i=1}^{k-1}{\mathfrak S}_{w^{(i)}}=0\impliedby {\sf code}(w^{(k)})\not\in \Newton\left(\prod_{i=1}^{k-1}{\mathfrak S}_{w^{(i)}}\right)=\sum_{i=1}^{k-1}{\mathcal S}_{D(w^{(i)})}.
\end{equation}
By Theorem~\ref{thm:genperm}, each ${\mathcal S}_{D(w^{(i)})}$ is a generalized permutahedron. Hence,
by Lemma~\ref{lem:minksum}, 
\[\Newton\left(\prod_{i=1}^{k-1}{\mathfrak S}_{w^{(i)}}\right)={\mathcal S}_{D'}.\]
Now we may apply Theorem~\ref{thm:ARY1} in the special case that $D=D'$ and $\alpha={\sf code}(w^{(k)})$ to
obtain the second sentence of the theorem. The final sentence follows from Corollary~\ref{cor:runtime}.
This completes the proof of
Theorem~\ref{thm:variation}.

The proof of Theorem~\ref{thm:main} is the same, except that we use Lemma~\ref{lemma:Oct2zzz}.

Theorem~\ref{thm:schubnec} follows from the above arguments, combined with Theorems~\ref{thm:ARY1}
and~\ref{thm:isSNP}.

Finally, we turn to Proposition~\ref{prop:Oct4yyy}.
We prove the contrapositive. Suppose Theorem~\ref{thm:variation}'s test is inconclusive, that is,
\begin{equation}
\label{eqn:Oct4}
[x^{{\sf code}(w^{(k)})}]{\mathfrak S}_{w^{(1)}}\cdots {\mathfrak S}_{w^{(k-1)}}\neq 0.
\end{equation}
\begin{claim} 
\label{claim:Oct19}
If $w\in S_n$ then 
${\sf code}(w)+{\sf code}(w_0w)=(n-1,n-2,\ldots,3,2,1,0)$.
\end{claim} 
\noindent\emph{Proof of Claim~\ref{claim:Oct19}:} By definition of $D(w)$,
\[c_r(w) =  (w(r)-1)-\#\{i<r:w(i)<w(r)\}.\]
On the other hand,
\begin{align*}
c_r(w_0w) & =(w_0w(r)-1)-\#\{i<r:w_0w(i)<w_0w(r)\}\\
				 \ & = \left((n+1-w(r))-1\right)-\#\{i<r:w(r)<w(i)\}.
\end{align*}
Hence $c_r(w)+c_r(w_0w)=n-r$, as desired.\qed

By
(\ref{eqn:Oct4}) and (\ref{eqn:hasterm}) combined, 
\begin{align*}
\ & [x_1^{n-1}x_2^{n-2}\cdots x_{n-1}]({\mathfrak S}_{w^{(1)}}\cdots {\mathfrak S}_{w^{(k-1)}}){\mathfrak S}_{w_0w^{(k)}}\\
 = & [x_1^{n-1}x_2^{n-2}\cdots x_{n-1}](x^{{\sf code}(w^{(k)})}+\cdots)(x^{{\sf code}(w_0w^{(k)})}+\cdots)\neq 0,
 \end{align*}
 where inequality is by Claim~\ref{claim:Oct19}. Thus Theorem~\ref{thm:main}'s test is inconclusive.
\qed

\subsection{A flexible version of the asymmetric test}
The condition (c) in defining ${\sf Tab}'$ can be replaced by the exponent vector of any monomial in ${\mathfrak S}_{w^{(k)}}$. Unfortunately, the number of such exponent vectors is potentially large. Instead,
one can sample points from ${\mathcal S}_{D(w)}$ as follows.
Construct the Rothe diagram $D(w)$. Fix a column $c$ of $D(w)$. Suppose the
boxes of $D(w)$ in that column are in rows $r_1,r_2,\ldots,r_z$. Find integers 
$1\leq x_1<x_2<\ldots<x_z$
such that $x_j\leq r_j$.  Repeat for every column $c$.
The result is an element of ${\sf Tab}(D(w),\alpha)$ for some $\alpha$.  (Thus one can create a randomized version of Theorem~\ref{thm:variation}.)

It is possible that, even with choice, no exponent vector exhibits nonvanishing:

\begin{example}
\label{exa:Oct8}
$C_{231645,231645}^{451623}=0$. Now, 
\[{\mathfrak S}_{451623}=x_1^{3}x_2^{3}x_4^{2}+x_1^{3}x_2^{3}{x_3}{ x_4}+x_1^{3}x_2^{3}x_3^{2}.\]
Here ${\sf code}(451623)=3302$. One can check that 
\[[x^{{\sf code}(451623)}]{\mathfrak S}_{231645}^2 > 0, [x_1^3x_2^3x_4^2]{\mathfrak S}_{231645}^2 >0, \text{\ and  
$[x_1^3x_2^3x_3 x_4] {\mathfrak S}_{231645}^2 >0$.}\]

Thus Theorem~\ref{thm:variation}'s test is inconclusive using any choice of monomial from
${\mathfrak S}_{451623}$.\qed
\end{example}

Individual monomials have no geometric meaning in Schubert calculus. Thus, our tests \emph{seem} inherently
combinatorial, as opposed to being avatars of the geometry.

\subsection{Certificate of vanishing}
Textbook linear programming results implying efficiency of Theorems~\ref{thm:main}
and~\ref{thm:variation} offer an additional benefit. There is a short certificate when ${\sf Tab}$ or ${\sf Tab}'$
is empty. This follows from standard reasoning using Farkas' lemma. 

Theorem~\ref{thm:schubnec} provides an alternative certification method. Recording one Schubitope
inequality defining ${\mathcal S}_D$ for which $(n,n-1,\ldots,2,1)$ fails proves $C_{w^{(1)},\ldots,w^{(k)}}=0$.
(A similar statement holds about ${\mathcal S}_{D'}$.) 

\section{Comparisons to other vanishing tests}
\label{sec:5}

We compare our tests to three non-\emph{ad hoc} vanishing tests. There are examples
where our method is successful where the others are not, and \emph{vice versa}.

\subsection{Bruhat order}
\emph{Bruhat order} on $S_n$ is (combinatorially) defined as the reflexive and transitive closure of the covering relations  
$u\leq ut_{ij}$ if $\ell(ut_{ij})=\ell(u)+1$, 
where $t_{ij}$ is the transposition interchanging $i$ and $j$. There exist efficient tests to determine $u\leq v$, such as the Ehresmann \emph{tableau criterion} \cite[Proposition~2.2.11]{Manivel}. The following is well-known; we include a proof since we do not know where it exactly
appears in the literature:

\begin{Fact}[Bruhat vanishing test]
\label{fact:bruhat}
$C_{w^{(1)},\ldots,w^{(k)}}=0$ if $w^{(i)}\not\leq w_0w^{(j)}$ for some $i\neq j$.
\end{Fact}
\begin{proof}
We prove the case $k=3$; the general case is similar. Say $u\not\leq w_0w$ but $C_{u,v,w}>0$. 
By Lemma~\ref{lemma:symmtoasymm},
$C_{u,v}^{w_0w}=C_{u,v,w}>0$. \emph{Monk's formula} \cite[Theorem~2.7.1]{Manivel} states that if $z\in S_n$,
\begin{equation}
\label{eqn:Monk}
\sigma_z\smile \sigma_{t_{m,m+1}}=\sum \sigma_{zt_{jk}} \in H^*(X);
\end{equation}
the sum is over all $j\leq m<k$ such that $\ell(zt_{jk})=\ell(w)+1$ and $zt_{jk}\in S_n$. 
Suppose $s_m:=t_{m,m+1}$ and $v=s_{m_1}s_{m_2}\cdots s_{m_{\ell(v)}}$ is a reduced expression for $v$.
By (\ref{eqn:Monk}), for some $\alpha\in {\mathbb Z}_{>0}$, 
\begin{equation}
\label{eqn:Oct22rrr}
\prod_{i=1}^{\ell(v)} \sigma_{s_{m_i}}=\alpha\, \sigma_{v}+\text{(positive sum of Schubert classes)}.
\end{equation}
By induction using (\ref{eqn:Monk}), 
\begin{equation}
\label{eqn:Oct23gtg}
[\sigma_y]\  \sigma_u\prod_{i=1}^{\ell(v)}\sigma_{s_{m_i}}\neq 0 \iff y\geq u.
\end{equation} 
By the positivity of Schubert calculus, and the assumption $C_{u,v}^{w_0w}>0$, 
\[[\sigma_{w_0w}] \  \sigma_u(\alpha\, \sigma_v+\text{(positive sum of Schubert classes)})\neq 0.\]
In view of (\ref{eqn:Oct22rrr}), this contradicts (\ref{eqn:Oct23gtg}).
\end{proof}

We give bad news first:

\begin{example}
$(u,v,w)=(1243,1342,3142)$ is a vanishing problem detected by Fact~\ref{fact:bruhat} since $1342=v\not\leq w_0w=2413$. Our methods do not detect $C_{u,v}^{w_0w}=C_{1243,1342}^{2413}$. Since 
\[{\mathfrak S}_{1243}{\mathfrak S}_{1342}=x_2 x_3^{2}+x_1 x_3^{2}+3 x_1 x_2 x_3+x_2^{2} x_3
+\underline{x_1 x_2^{2}}+x_1^{2} x_3+\underline{x_1^{2} x_2},\]
contains both monomials of ${\mathfrak S}_{w_0w}={\mathfrak S}_{2413}=x_1 x_2^{2}+x_1^{2} x_2$, no monomial
of ${\mathfrak S}_{w_0w}$ can be used to detect vanishing. In particular, Theorem~\ref{thm:variation} is inconclusive
(and hence by Proposition~\ref{prop:Oct4yyy}, the symmetric test is
also inconclusive.) Since $C_{u,v,w}=C_{v,w}^{w_0u}=C_{u,w}^{w_0v}$, one hopes the asymmetric method
shows either $C_{v,w}^{w_0u}=C_{1342,3142}^{4312}=0$ or $C_{u,w}^{w_0v}=C_{1243,3142}^{4213}=0$. Unfortunately, both attempts are similarly inconclusive.\qed
\end{example}

\begin{example}
The vanishing of the Schubert problem 
$(u,v,w)=(1423,1423,1423)$
is undetected by Fact~\ref{fact:bruhat}. Now
\[{\mathfrak S}_{1423}^3=x_2^{6}+3 x_1 x_2^{5}+6 x_1^{2} x_2^{4}+7 x_1^{3} x_2^{3}+6 x_1^{4} x_2^{2}
+3 x_1^{5} x_2+x_1^{6}\]
does not contain 
${\mathfrak S}_{w_0}={\mathfrak S}_{4321}=x_1^3 x_2^2 x_3$
and hence vanishing is
seen by Theorem~\ref{thm:main}.\qed
\end{example}

\subsection{A.~Knutson's descent cycling} In \cite{Knutson:descent}, A.~Knutson introduced a vanishing criterion.  
Recall, $u\in S_n$ has a \emph{descent} at position $i$ if $u(i)>u(i+1)$ and has an \emph{ascent} at position $i$ otherwise.
That is, respectively, $us_i\leq u$ and $us_i\geq u$.

\begin{Fact}[dc triviality]
If $(u,v,w)$ is a Schubert problem such that 
$us_i\geq u, vs_i\geq v, ws_i\geq w$
then $C_{u,v,w}=0$.
\end{Fact}

\begin{example}
\label{baddc}
The triple $(1423,1423,1342)$ is dc trivial and hence 
$C_{1423,1423,1342}=0$. Here, the asymmetric
test (Theorem~\ref{thm:variation}) is inconclusive (again, thus by Proposition~\ref{prop:Oct4yyy}, the symmetric test is
also inconclusive). 
Indeed, $C_{u,v}^{w_0w}=C_{1423,1423}^{4213}=0$ is not detected since 
\[{\mathfrak S}_{1423}^2=x_2^{4}+2 {x_1} x_2^{3}+3 x_1^{2} x_2^{2}+\underline{2 x_1^{3} {x_2}}+x_1^{4},\]
but
${\mathfrak S}_{w_0w}={\mathfrak S}_{4213}=x_1^3  x_2$. Also $C_{u,w}^{w_0v}=0$ and $C_{v,w}^{w_0u}=0$
are not detected since
\[{\mathfrak S}_{1423}{\mathfrak S}_{1342}=x_2^{3} {x_3}+2 {x_1} x_2^{2} {x_3}+2 x_1^{2} {x_2} {x_3}+
\underline{x_1^{3} {x_3}}+{x_1} x_2^{3}+x_1^{2} x_2^{2}+\underline{x_1^{3} {x_2}}.\]
Since
${\mathfrak S}_{w_0u}={\mathfrak S}_{w_0v}={\mathfrak S}_{4132}=x_1^{3} {x_3}+x_1^{3} {x_2}$,
no lattice point in ${\mathcal S}_{D(4132)}$ proves vanishing.
\qed
\end{example}

\begin{example}
\label{exa:Oct9aaa}
The Schubert problem $(3256147, 2143657, 4632175)$ from Example~\ref{first:exa} is not
dc trivial, but 
$C_{3256147,2143657,4632175}=0$, 
as determined by Theorem~\ref{thm:main}.\qed
\end{example}

Define the \emph{descent cycling equivalence} $\sim$  on
Schubert problems by
\begin{itemize}
\item[(dc.1)] $(u,v,w)\sim (us_i,v,ws_i), (u,vs_i,ws_i)$ if $us_i\geq u,vs_i\geq v, ws_i\leq w$;
\item[(dc.2)] $(u,v,w)\sim (us_i,v,ws_i), (us_i,vs_i,w)$ if $us_i\leq u,vs_i\geq v, ws_i\geq w$;
\item[(dc.3)] $(u,v,w)\sim (u,vs_i,ws_i),(us_i,vs_i,w)$ if $vs_i\leq v, us_i\geq u,ws_i\geq w$.
\end{itemize}

\begin{Fact}[{\cite{Knutson:descent}}]
\label{fact:equiv0}
$C_{u,v,w}=C_{u',v',w'}$ if $(u,v,w)\sim(u',v',w')$. In particular,
$C_{u,v,w}=0$ if $(u,v,w)$
 is $\sim$ equivalent to a dc trivial problem.
\end{Fact}

\begin{example}
A reported in \cite{Knutson:descent}, for $n=6$ there is one dc equivalence class of problems $(u,v,w)$ which vanishes but does not contain a dc trivial triple.
This is precisely the problem studied in Example~\ref{exa:Oct8}, which our methods also cannot explain. \qed
\end{example}

\begin{example}
Let $(u,v,w)=(3216547, 3216547, 4261573)$ be a problem in $S_7$. Theorem~\ref{thm:main} shows 
$C_{u,v,w}=0$ (any element of ${\sf Tab}$ must contain at least seven $1$'s). The $\sim$
class contains $9$ elements, namely
\[(3216574,  3261547, 4216537), (3216547, 3216574, 4261537), (3261547, 3216574, 4216537),\] 
\[(3261547, 3216547, 4216573), (3216574, 3216547, 4261537), (3216547, 3216547, 4261573),\]
\[(3261574, 3216547, 4216537), (3216547, 3261574, 4216537), (3216547, 3261547, 4216573).\]
None are dc trivial and thus Fact~\ref{fact:equiv0} is inconclusive.\qed 
\end{example}

\subsection{K.~Purbhoo's root games} K.~Purbhoo's \emph{root games} from \cite{Purbhoo:root} give a vanishing
criteria. Fix the positive roots $\Phi^+$ associated
to $GL_n$ to be 
$\alpha_{i,j}=\varepsilon_i-\varepsilon_j \text{\ for $1\leq i<j\leq n$,}$
where $\varepsilon_i$ is the $i$-th
standard basis vector. The poset $P$ of positive roots takes the form
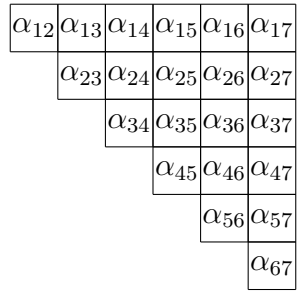
\[\tableau{\alpha_{12} & \alpha_{13} & \alpha_{14} & \alpha_{15} & \alpha_{16} & \alpha_{17}\\
 & \alpha_{23} & \alpha_{24} & \alpha_{25} & \alpha_{26} & \alpha_{27}\\
 && \alpha_{34} & \alpha_{35} & \alpha_{36} & \alpha_{37}\\
 &&&\alpha_{45} & \alpha_{46} & \alpha_{47}\\
 &&&&\alpha_{56} & \alpha_{57}\\
 &&&&& \alpha_{67}}
\]
The maximal element of this poset is the highest root $\alpha_{1n}$.
For each $i$ place a token $\bullet$ in square $\alpha_{mn}$ if $w^{(i)}(m)>w^{(i)}(n)$. 
This is called the \emph{initial position}. An upper order filter $A$ is an up-closed subset of $P$. 
This initial position is \emph{doomed} if there exists an upper order filter $A$ such that there are more
tokens in $A$ than $\#A$. This is \cite[Theorem~3.6]{Purbhoo:root}:
\begin{Fact}[Doomed root game]
\label{fact:purbhoo}
If $(w^{(1)},\ldots,w^{(k)})$'s initial position is doomed, $C_{w^{(1)},\ldots,w^{(k)}}=0$.
\end{Fact}

This test is quite handy. However, the number of upper order filters for type $A_{n-1}$ is the Catalan number $C_{n}=\frac{1}{n+1}{2n\choose n}$, which is exponential in $n$. 

\begin{example}
The vanishing of $(1423,1423,1342)$ is seen by 
Fact~\ref{fact:purbhoo}. This is doomed:
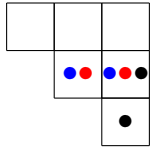
\[\tableau{ \ & \ & \ \\
 & {\color{blue} \bullet}{\color{red} \bullet} & {\color{blue} \bullet}{\color{red} \bullet}\bullet \\
 && \bullet  }
\]
As is explained in Example~\ref{baddc}, our methods are inconclusive here.\qed
\end{example}

\begin{example}
Let $u=v=3216547$ and $w=1652473$. Below we mark the inversions of $u,v,w$ with ${\color{blue} \bullet},{\color{red} \bullet}, \bullet$ respectively.
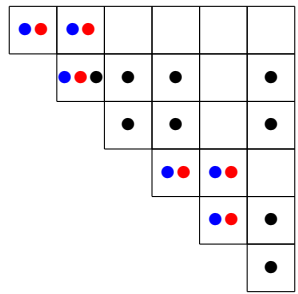
\[\tableau{ {\color{blue} \bullet}{\color{red} \bullet}& {\color{blue} \bullet} {\color{red} \bullet} & \ & \ & \  & \ \\
 & {\color{blue} \bullet}{\color{red} \bullet} \bullet &\bullet & \bullet & \ & \bullet  \\
 && \bullet & \bullet &\  &\bullet  \\
 &&& {\color{blue} \bullet}{\color{red} \bullet} &  {\color{blue} \bullet}{\color{red} \bullet} &\ \\
 &&&& {\color{blue} \bullet}{\color{red} \bullet} & \bullet \\
 &&&&& \bullet }
\]
This game is not doomed, so Fact~\ref{fact:purbhoo} is inconclusive here.
(Descent cycling doesn't help either, as the equivalence class of size $9$ contains no dc trivial elements.)
Also, Theorem~\ref{thm:main} does not succeed.
However, Theorem~\ref{thm:variation}'s test shows $C_{u,v}^{w_0w}=C_{3216547,3216547}^{7236415}=0$. \qed
\end{example}

\section*{Acknowledgements}
We thank Husnain Raza for writing code, as part of the Illinois Combinatorics Lab for Undergraduate Experience (ICLUE) program, to help study Allen Knutson's descent cycling. We thank our collective coauthors Anshul Adve, Alex Fink, Karola M\'esz\'aros, Cara Monical, Neriman Tokcan, and Colleen Robichaux from \cite{Fink.Meszaros.StDizier, MTY, ARY} for their work upon which is paper is possible.
AS was supported by an NSF postdoctoral fellowship. AY was partially supported by an Simons Collaboration Grant, an NSF RTG 1937241 in Combinatorics, and an appointment at the UIUC Center for Advanced Study.

\end{document}